\documentclass[12pt, reqno]{amsart}
\usepackage{pifont}
\usepackage{mathrsfs}
\usepackage{geometry}
\usepackage{titletoc}
\usepackage{stix2}

\usepackage{amsmath}
\usepackage{amssymb} 
\usepackage{enumitem} 
\usepackage{mathtools} 
\usepackage[table]{xcolor} 
\usepackage[all]{xy} 
\usepackage{tikz} 
\usepackage{tikz-cd}
\usepackage{indentfirst} 
\usepackage{babel} 
\usepackage{setspace} 

\usepackage[colorlinks,linkcolor=red,anchorcolor=green,citecolor=blue]{hyperref} 
\hypersetup{linktocpage = true} 

\usepackage{rotating} 

\usepackage{ytableau} 
\usepackage{longtable} 
\newcolumntype{M}[1]{>{\centering\arraybackslash}m{#1}} 

\usepackage{fancyhdr}

\newcommand\bp{{\bar\partial}}

\theoremstyle{plain}
\newtheorem{thm}{Theorem}[section]
\newtheorem{lemma}[thm]{Lemma}
\newtheorem{prop}[thm]{Proposition}
\newtheorem{cor}[thm]{Corollary}
\newtheorem{defn}[thm]{Definition}

\theoremstyle{definition}
\newtheorem{example}[thm]{Example}
\newtheorem{remark}[thm]{Remark}

\newcommand{\btheorem}{\begin{thm}}
    \newcommand{\etheorem}{\end{thm}}
\newcommand{\bproposition}{\begin{prop}}
    \newcommand{\eproposition}{\end{prop}}
\newcommand{\bdefinition}{\begin{defn}}
    \newcommand{\edefinition}{\end{defn}}
\newcommand{\bcorollary}{\begin{cor}}
    \newcommand{\ecorollary}{\end{cor}}
\newcommand{\bproof}{\begin{proof}}
    \newcommand{\eproof}{\end{proof}}
\newcommand{\bremark}{\begin{remark}}
    \newcommand{\eremark}{\end{remark}}
\newcommand{\eexample}{\end{example}}
\newcommand{\bexample}{\begin{example}}

\newcommand{\elemma}{\end{lemma}}
\newcommand{\blemma}{\begin{lemma}}

\newcommand{\la}{\langle}
\newcommand{\ra}{\rangle}
\newcommand{\sq}{\sqrt{-1}}

\newcommand{\p}{\partial}

\renewcommand{\bar}{\overline}
\newcommand{\eps}{\varepsilon}

\renewcommand{\phi}{\varphi}

\newcommand{\beq}{\begin{equation}}
\newcommand{\eeq}{\end{equation}}
\newcommand{\ee}{\end{eqnarray*}}
\newcommand{\be}{\begin{eqnarray*}}

\newcommand{\bd}{\begin{enumerate}}
    \newcommand{\ed}{\end{enumerate}}

\renewcommand{\hat}{\widehat}
\renewcommand{\tilde}{\widetilde}

\newcommand{\qtq}[1]{\quad\mbox{#1}\quad}
\renewcommand{\bp}{\bar{\partial}}
\newcommand{\Om}{\Omega}

\newcommand{\ts}{\otimes}

\renewcommand{\S}{{\mathbb S}}

\renewcommand{\>}{\rightarrow}

\newcommand{\C}{{\mathbb C}}

\newcommand{\R}{{\mathbb R}}
\newcommand{\diam}{\mathrm{diam}}

\newcommand{\LL}{\left\langle}
\newcommand{\RL}{\right\rangle}
\newcommand{\MS}{\mathrm{SB}}

\newcommand{\MLC}{\mathrm{LC}}
\newcommand{\Ric}{\mathrm{Ric}}

\newcommand{\bpzi}{\frac{\p}{\p \bar z^i}}

\newcommand{\bpzl}{\frac{\p}{\p \bar z^\ell}}

\newcommand{\pzi}{\frac{\p}{\p z^i}}

\newcommand{\pzl}{\frac{\p}{\p z^\ell}}

\newcommand{\pxi}{\frac{\p}{\p x^i}}

\newcommand{\pxl}{\frac{\p}{\p x^\ell}}

\newcommand{\sn}{\mathrm{sn}}

\renewcommand{\>}{\rightarrow}

\renewcommand{\p}{{\partial}}
\renewcommand{\bp}{{\bar{\partial}}}

\setlist[itemize]{leftmargin=*}
\setlist[enumerate]{leftmargin=*}

\numberwithin{equation}{section} 

\makeatletter

\usepackage{fancyhdr}
\title{Comparison theorems in Hermitian geometry I}

\author{Xiaokui Yang}

\address{Xiaokui Yang, Department of Mathematics and Yau Mathematical Sciences Center, Tsinghua University, Beijing, 100084, China}
\email{xkyang@mail.tsinghua.edu.cn}

\begin{document}

    \begin{abstract}This paper develops  second variational formulas and index forms in the context of Hermitian geometry. 
    Building upon these analytical foundations, we establish results analogous to classical theorems in Riemannian geometry, including Myers' theorem,  Laplacian comparison theorems and volume comparison theorems.
   \end{abstract}

   \maketitle

   \section{Introduction}

   Let $(M, g, J)$ be a Hermitian manifold, where $g$ is a Riemannian metric and $J \in \mathrm{End}(T_\mathbb{R} M)$ is an integrable complex structure compatible with $g$. In the study of Hermitian geometry, several canonical metric compatible affine connections play fundamental roles.  The most prominent connections include: \bd \item The Chern connection $\nabla^{\mathrm{Ch}}$: The unique connection preserving both the Hermitian metric  and the holomorphic structure of $M$.

    \item The Strominger-Bismut connection $\nabla^{\mathrm{SB}}$(\cite{Str86, Bis89}): A geometrically natural connection that emerges in heterotic string theory and plays a crucial role in the study of complex non-K\"ahler geometry.

    \item The restricted Levi-Civita connection $\hat\nabla^{\mathrm{LC}}$: The projection of the complexified Levi-Civita connection $\nabla^{\mathrm{LC}}$ to the holomorphic tangent bundle $T^{1,0}M$. \ed
   Remarkably, when $(M, g, J)$ is K\"ahler, these connections coincide, reflecting the special geometric properties of K\"ahler manifolds. Recent research has significantly advanced our understanding of these connections, particularly in applications to string theory and the study of non-K\"ahler Calabi-Yau manifolds. For a more comprehensive treatment of related topics, we refer  readers to  the works of
    \cite{AD99, IP01, FG04, LY05, Tos07, FY08, FTY09, Fu10,  Li10,  LY12, Zha12,ACS15,
    FY15, AU16, LY16,WYZ16, Yang16, LU17,  OUV17, Yu18, AOUV18, CZ18,  HLY18, NZ18, YZ18a, YZ18b,  FZ19, CCN19, FZ19, Yang19, YZZ19, ZZ19a, Yang20, NZ23, WY25} and the references therein.\\

    In this paper, we establish new extensions of fundamental results from Riemannian geometric analysis to the framework of Hermitian geometry, with particular emphasis on the Strominger-Bismut connection. Our work bridges important gaps between these two geometric theories while revealing new phenomena  to the Hermitian setting.  Let $(M,g, J)$ be a Hermitian manifold and let $\omega_g$ denote the fundamental
    $2$-form associated with the Hermitian structure $(g, J)$:
    $$ \omega_g(X,Y)=g(JX,Y)$$
    for $X, Y\in T_\R M$.  It is well-known that the Strominger-Bismut connection $\nabla^{\mathrm{SB}}$ on the real tangent bundle $T_{\mathbb{R}}M$ of a Hermitian manifold $(M,g,J)$ is characterized by the following relation:
     \beq
    g(\nabla_X^{\MS}Y,Z)=g(\nabla_X^{\MLC}Y,Z)+\frac{1}{2}\left(d\omega_g\right)(JX,JY,JZ), \label{defnSB}
    \eeq
 where $\nabla^{\MLC}$ is the Levi-Civita connection of $(M,g)$.  A key feature of this connection is that its torsion tensor   \beq  T^{\MS}(X,Y,Z):=g(\nabla^\MS_XY-\nabla^\MS_YX-[X,Y],Z)\label{torsion}\eeq
   is totally skew symmetric, i.e. $T^\MS \in \Om^3(M)$. This property makes it particularly significant in complex non-K\"ahler geometry and theoretical physics. Moreover, since $T^\MS$ is totally skew-symmetric,  for any $X\in\Gamma(M,TM)$,  one can seee clearly that
   \beq
   \nabla^{\text{SB}}_XX=\nabla^{\text{LC}}_XX.
   \eeq
   In particular, if $\gamma$ is a geodesic  with respect to the Levi-Civita connection $\nabla^{\text{LC}}$, then it is also a geodesic with respect to the Strominger-Bismut connection $\nabla^{\mathrm{SB}}$. Consequently, we will not distinguish between these connections when discussing geodesics in subsequent analysis.

   The first main result establishes the second variation formula for the energy functional of unit speed geodesics in a Hermitian manifold $(M, g, J)$. This extends classical Riemannian variational principles to the complex-geometric setting, where the Strominger-Bismut connection  $\nabla^{\MS}$
   plays a crucial role in characterizing the torsion-modified curvature effects.

    \begin{thm}\label{thm-2nd}
    Let $(M,\omega_g)$ be a Hermitian mainfold and $\gamma(t):[a,b]\to M$ be a unit speed geodesic. Then the second variation of the energy of $\gamma$ is
    \begin{eqnarray}
    \frac{\p^2  E(\alpha)}{\p s_1\p s_2}\bigg|_{s_1=s_2=0}
    \nonumber&=&\int_a^b\left\{\LL\hat{\nabla}^{\MS}_{\frac{d}{d t}}V,\hat{\nabla}^{\MS}_{\frac{d}{d t}}W\RL+T^{\MS}\left(V,\gamma',\hat{\nabla}^{\MS}_{\frac{d}{d t}}W\right) dt-R^{\MS}(V,\gamma',\gamma',W)\right\}dt\\
    &&+\left.\LL\left.\left(\bar{\nabla}^\MS_{\frac{\p}{\p s_1}}\alpha_*\left(\frac{\p}{\p s_2}\right)\right)\right|_{s_1=s_2=0},\gamma'\RL\right|_{t=a}^{t=b},\label{2nd}
    \end{eqnarray}
    where  $V$ and $W$ are variational vector fields of $\alpha$, with
    $\hat{\nabla}^{\mathrm{SB}}$ and $\bar{\nabla}^{\mathrm{SB}}$ representing the connections
    induced on  $\gamma^*TM$ and $\alpha^*TM$ respectively via the
    Strominger-Bismut connection $\nabla^{\mathrm{SB}}$ on $(TM, g, J)$. Moreover,  if $\alpha$ is a proper variation of $\gamma$, then \beq
    \frac{\p^2  E(\alpha)}{\p s_1\p s_2}\bigg|_{s_1=s_2=0}=\int_a^b\left\{\LL\hat{\nabla}^{\MS}_{\frac{d}{d t}}V,\hat{\nabla}^{\MS}_{\frac{d}{d t}}W\RL+T^{\MS}\left(V,\gamma',\hat{\nabla}^{\MS}_{\frac{d}{d t}}W\right)-R^{\MS}(V,\gamma',\gamma',W)\right\}dt. \label{normalsecondvariation}
    \eeq
   \end{thm}

\noindent
While the energy functional $E(\alpha)$  and its second variation  $\frac{\p^2  E(\alpha)}{\p s_1\p s_2}$ are intrinsically determined by the Riemannian metric
$g$, the explicit form of the second variation formula depends crucially on the choice of connection. In Hermitian geometry, this leads to multiple meaningful expressions:
\bd
    \item The Levi-Civita connection yields the classical Riemannian formulation;
    \item The Chern connection provides a natural complex-geometric perspective;
    \item The Strominger-Bismut connection reveals torsion effects in the variational structure.
\ed
Our approach utilizes the Strominger-Bismut connection to highlight the interplay between metric variations and the Hermitian torsion. In Riemannian geometry, the index form $I_\gamma: \Gamma(\gamma^*TM)\times \Gamma(\gamma^*TM)\>\R$ is classically defined using the Levi-Civita connection  $\nabla^\MLC$:
\beq I_\gamma(V,W) =\int_a^b\left\{\LL\hat{\nabla}^{\text{LC}}_{\frac{d}{dt}}V,\hat{\nabla}^{\text{LC}}_{\frac{d}{dt}}W\RL-R^{\text{LC}}(V,\gamma',\gamma',W)\right\}dt. \eeq
We   reformulate it by using  the Strominger-Bismut connection $\nabla^\MS$.

    \btheorem\label{thm-index form}     Let $(M,\omega_g)$ be a Hermitian mainfold and $\gamma(t):[a,b]\to M$ be a unit speed geodesic. Then for any smooth vector fields $V$ and $W$ along $\gamma$, the index form can be written as
   \be I_\gamma(V,W)&=&\int_a^b\left\{\LL\hat{\nabla}^{\MS}_{\frac{d}{d t}}V,\hat{\nabla}^{\MS}_{\frac{d}{d t}}W\RL+T^{\MS}\left(V,\gamma',\hat{\nabla}^{\MS}_{\frac{d}{d t}}W\right)-R^{\MS}(V,\gamma',\gamma',W)\right\}dt\\
   &&+\left.\frac{1}{2}T^\MS(V,W,\gamma')\right|_{t=a}^{t=b}.\ee
   In particular,
   \beq I_\gamma(V,V)=\int_a^b\left\{\LL\hat{\nabla}^{\MS}_{\frac{d}{d t}}V,\hat{\nabla}^{\MS}_{\frac{d}{d t}}V\RL+T^{\MS}\left(V,\gamma',\hat{\nabla}^{\MS}_{\frac{d}{d t}}V\right)-R^{\MS}(V,\gamma',\gamma',V)\right\}dt.\eeq
   \etheorem

   \noindent The index form constitutes a fundamental analytical tool in geometric variational theory, with particularly profound applications in comparison theorems for both Riemannian and K\"ahler manifolds. As applications of Theorem \ref{thm-2nd} and Theorem \ref{thm-index form}, we derive results analogous to classical theorems in Riemannian geometry, including Myers’ theorem and various comparison theorems. Recall that
   the \emph{real Ricci curvature of the Strominger-Bismut connection} is defined to be
   \beq
   \Ric^{\MS}(X,Y):=\sum_{i,\ell=1}^{2n}g^{i\ell} R^{\MS}\left(\frac{\p}{\p x^i},X,Y,\frac{\p}{\p x^\ell}\right), \label{realSB}
   \eeq
   where  $X,Y\in T_\R M$.
    It is well-known that the first Bianchi identity fails to hold for  $R^\MS$, resulting in  $ \Ric^{\MS}$
    not necessarily being symmetric.  It is well-known that  Ricci curvature tensors derived from the Levi-Civita connection and the Strominger-Bismut connection possess intrinsically different geometric properties, making them generally incomparable. For detailed discussions on the comparative analysis of these and other Ricci curvatures, we refer to \cite{LY12}, \cite{LY16}, and \cite{WY25}.
     In the Hermitian setting,  the holomorphic Ricci curvature serves as the complex analogue of the real Ricci curvature and is defined by  \beq   \mathfrak{Ric}^\MS(V,\bar V):=\sum_{i,\ell=1}^nh^{i\bar\ell}R^\MS\left(\frac{\p}{\p z^i},
    \bar V, V, \frac{\p}{\p {{\bar z}^\ell}}\right)\eeq 
    where $V\in\Gamma(M,T^{1,0}M)$.  The following theorem is a Hermitian counterpart to Myers' theorem by using the holomorphic  Ricci curvature of the Strominger-Bismut connection $\nabla^\MS$.

    \btheorem\label{Myers}  Let $(M,\omega)$ be a complete balanced Hermitian manifold of complex dimension $n$. Suppose that there exists some constant $K>0$ such that the holomorphic Ricci curvature of the Strominger-Bismut connection satisfies
 \beq \mathfrak{Ric}^{\MS} (V,\bar V)\geq (2n-1)K |V|^2, \label{complexcurvature} \eeq 
 for each $V\in\Gamma(M,T^{1,0}M)$. \bd \item $M$ is compact and 
   $$
   \diam(M, \omega)\leq \frac{\pi}{\sqrt{K}}.
   $$
   \item The fundamental group of $M$ is finite.
  \item The volume comparison theorem holds:
 \beq  \mathrm{Vol}(M,\omega)\leq \mathrm{Vol}\left(\mathbb S^{2n}\left(1/\sqrt{K}\right), g_{\mathrm{can}}\right).\eeq 
   \ed 
   \etheorem
\noindent The proof of Theorem \ref{Myers} demonstrates that both the Laplacian comparison theorem and the local volume comparison theorem remain valid under the curvature condition \eqref{complexcurvature}. In the K\"ahler case ($\omega$ K\"ahler), Theorem \ref{Myers} is classical, and the manifold $M$ is known to be simply connected. However, when 
$\omega$ is merely balanced, the question of whether $M$ retains simple connectivity remains open. On the other hand,   using perturbation methods  (e.g. \cite{LY16} and \cite{WY25}), one can easily construct  non-K\"ahler Hermitian  metrics satisfying $\mathfrak{Ric}^{\MS}>0$ while their Levi-Civita Ricci curvature need not be positive.‌\\

   Holomorphic sectional curvature is also a fundamental concept in complex geometry that  carries rich geometric information, sharing important relationships with both Ricci curvature and holomorphic bisectional curvature   of complex manifolds. The holomorphic sectional curvature $\mathrm{HSC}^{\MS}$ of the Strominger-Bismut connection $\nabla^\MS$ is defined as
  \beq
  \mathrm{HSC}^{\MS}(X):=R^{\MS}(JX,X,X,JX)/|X|^4. \eeq
 The following result demonstrates a structural analogy to the foundational theorems established by Sygne and Tsukamoto (\cite{Tsu57}), extending their framework within the Hermitian geometric setting.
   \begin{thm}\label{thm-compactness from HSC}
    Let $(M,\omega)$ be a complete Hermitian manifold. Suppose that the holomorphic sectional curvature  of the Strominger-Bismut connection $\nabla^\MS$ satisfies $$  \mathrm{HSC}^{\MS}\geq K$$  for some constant  $K>0$.  Then  $M$ is compact and
    $$
    \diam(M,\omega)\leq \frac{\pi}{\sqrt{K}}.
    $$
Moreover,  $M$ is simply connected.
   \end{thm}

\noindent In particular, we establish the following result, which parallels Weinstein's theorem on positive Riemannian sectional curvature and extends the K\"ahler setting in \cite{NZ18}.
   \btheorem\label{cor-auto}
    Let $(M,\omega_g)$ be a compact Hermitian manifold with positive holomorphic sectional curvature  $\mathrm{HSC}^{\MS}>0$. Then every isometry $F:(M,g)\>(M,g)$ has a fixed point.
   \etheorem

 \noindent\textbf{Acknowledgements}.  The author would like to Valentino Tosatti, Bo Yang, Kaijie Zhang and Fangyang Zheng  for their insightful discussions.   Special thanks go to  Zhitong Chen for  verifying many details in this paper.

 \vskip 2\baselineskip

   \section{Curvature notions}

\noindent   Let $(M,g,J)$ be a Hermitian manifold. There is a unique affine connection $$\nabla: \Gamma(M,T_\R M)\times  \Gamma(M,T_\R M)\> \Gamma(M,T_\R M) $$ satisfying
   \bd \item $\nabla$ is compatible with $g$ and $J$;
   \item The  torsion tensor $$ T(X,Y,Z):=g(\nabla_XY-\nabla_YX-[X,Y],Z)$$  is totally anti-symmetric, i.e. $T\in \Om^3(M)$.
   \ed
   This connection is called the \textbf{Strominger-Bismut connection} and it is denoted by $\nabla^{\MS}$. Let $\omega_g$ denote the fundamental
  $2$-form associated with the Hermitian structure:
   $$ \omega_g(X,Y)=g(JX,Y),\ \ X, Y\in T_\R M.$$

\noindent The following foundational formulation will be employed throughout our work:
   \blemma\label{definition of Strominger connection} The Strominger connection $\nabla^\MS$ on  $(M,g, J)$ is given by
   \beq
   g(\nabla_X^{\MS}Y,Z)=g(\nabla_X^{\MLC}Y,Z)+\frac{1}{2}\left(d\omega_g\right)(JX,JY,JZ),  \label{defnSB1}
   \eeq where $\nabla^{\MLC}$ is the Levi-Civita connection on $(M,g)$ and  the torsion of $\nabla^\MS$ is
   \beq
   T^\MS(X,Y,Z)=\left(d\omega_g\right)(JX,JY,JZ).\label{torsion1}
   \eeq
   \elemma
   \bproof  It follows from a straightforward computation. \eproof 

   \noindent The \emph{ curvature tensor} $R^\MS$ of the Strominger-Bismut connection $\nabla^\MS$ is defined as
   \beq
   R^{\MS}(X,Y,Z,W)=\LL\nabla^{\MS}_X\nabla^{\MS}_YZ-\nabla^{\MS}_Y\nabla^{\MS}_XZ-\nabla^{\MS}_{[X,Y]}Z,W\RL.
   \eeq

\noindent  The following lemma describes the essential curvature property of  Strominger-Bismut connection $\nabla^{\MS}$:   \begin{lemma}
    For any $X,Y,Z,W\in\Gamma(M,T_\C M)$, it holds that
    \beq
    R^{\MS}(X,Y,Z,W)=-R^{\MS}(Y,X,Z,W)=-R^{\MS}(X,Y,W,Z).\label{skewsymmetry}
    \eeq
   \end{lemma}
   \begin{proof} The first identity follows from the definition. The second identity can be derived from  the relation $    R^{\MS}(X,Y,Z,Z)=0$.
    Indeed,
    \be
    &&  R^{\MS}(X,Y,Z,Z)=\LL\nabla^{\MS}_X\nabla^{\MS}_YZ-\nabla^{\MS}_Y\nabla^{\MS}_XZ-\nabla^{\MS}_{[X,Y]}Z,Z\RL\\
    &=&X\LL\nabla^{\MS}_YZ,Z\RL-\LL\nabla^{\MS}_YZ,\nabla^{\MS}_X Z\RL-Y\LL\nabla^{\MS}_XZ,Z\RL+\LL\nabla^{\MS}_XZ,\nabla^{\MS}_Y Z\RL-\frac{1}{2}[X,Y](|Z|^2)\\
    &=&X\left(\frac{1}{2}Y(|Z|^2)\right)-Y\left(\frac{1}{2}X(|Z|^2)\right)-\frac{1}{2}[X,Y]|Z|^2=0.
    \ee
    The proof is completed.
   \end{proof}

\noindent The curvature tensor $R^{\MS}$
exhibits key deviations from  the symmetric properties of the Riemannian curvature tensor:
\bd \item   the first Bianchi identity fails to hold for $R^{\MS}$. 
\item    $R^{\MS}(X,Y,Z,W)\neq R^{\MS}(Z,W,X,Y)$ in general.
\ed

   \blemma
    For any $X\in\Gamma(M,T_\R M)$, it holds that
    \beq
    \nabla^{\MS}_XX=\nabla^{\MLC}_XX.
    \eeq
    In particular, if $\gamma$ is a geodesic on $M$, then $\hat{\nabla}^{\MS}_{\frac{d}{dt}}\gamma'=0$.
   \elemma

   \blemma
    Let $\gamma:[a,b]\to M$ be a smooth curve and
    \beq
    E(\gamma):=\frac{1}{2}\int_a^b|\gamma'(t)|^2dt
    \eeq
    be the energy functional. Then for any proper variation $\alpha:[a,b]\times(-\eps,\eps)\to M$ of $\gamma$ with variational vector field $V$  it holds that
    \beq
    \frac{d}{d s}\bigg|_{s=0}E(\alpha(\bullet,s))=-\int_a^b\LL V,\hat{\nabla}^\MS_{\frac{d}{dt}}\gamma'\RL dt.
    \eeq
   \elemma
   \begin{proof}
    It follows from the well-known result
    \beq
    \frac{d}{d s}\bigg|_{s=0}E(\alpha(\bullet,s))=-\int_a^b\LL V,\hat{\nabla}_{\frac{d}{dt}}^{\text{LC}}\gamma'\RL dt
    \eeq
    and the fact that $\hat{\nabla}^\MS_{\frac{d}{dt}}\gamma'=\hat{\nabla}^{\MLC}_{\frac{d}{dt}}\gamma'$.
   \end{proof}

   Let $(M,g,J)$ be a Hermitian manifold.
   Let $\gamma:[a,b]\to(M,g)$ be a smooth curve and $v\in T_{\gamma(a)}M$. Then there exists a unique vector field $V$,
   called the \emph{Strominger-Bismut parallel vector field along $\gamma$}, such that $\hat{\nabla}_{\frac{d}{dt}}^\text{SB}V\equiv 0$ along $\gamma$ and $V(a)=v$.
   For $t_0,t\in [a,b]$,  we define the map
   \beq
   P_{t_0,t;\gamma}^{\text{SB}}:T_{\gamma(t_0)}M\to T_{\gamma(t)}M
   \eeq
   by $P_{t_0,t;\gamma}^{\text{SB}}v_0=V(t)$ where $V$ is the unique Strominger-Bismut parallel vector field along $\gamma$ satisfying $V(t_0)=v_0$.
   The collection of maps $P_{t_0,t;\gamma}^{\text{SB}}$ for $t_0,t\in [a,b]$ is called a \emph{Strominger-Bismut parallel transport along $\gamma$}.\\

For readers' convenience, we recall  local computations on Hermitian manifolds.     Let $(M,J)$ be a complex manifold of complex dimension  $n$. By
Newlander-Nirenberg's theorem, there exists a real coordinate system
$\{x^i,x^I\}$ where
$$
\{x^i\} \qtq{for} 1\leq i\leq n;\ \ \ \ \  \{x^I\} \qtq{for} n+1\leq
I\leq 2n \qtq{and}\ \ \ I=i+n.$$
such that $z^i=x^i+\sq x^I$ are  local holomorphic
coordinates on $M$ and
$$
J\left(\frac{\p}{\p x^i}\right)=\frac{\p}{\p x^I} \qtq{and}
J\left(\frac{\p}{\p x^I}\right)=-\frac{\p}{\p x^i}.$$
For $1\leq i\leq n$, $ dz^i=dx^i+\sq dx^I$, $d\bar z^i:=dx^i-\sq dx^I$
and  $$ \frac{\p}{\p
	z^i}=\frac{1}{2}\left(\frac{\p}{\p x^i}-\sq \frac{\p}{\p
	x^I}\right),\ \ \ \frac{\p}{\p\bar
	z^i}=\frac{1}{2}\left(\frac{\p}{\p x^i}+\sq \frac{\p}{\p
	x^I}\right).$$
Let $(M,g, J)$ be a Hermitian manifold and  $T_\C M=T_\R M\ts \C$
be the complexification. One can extend  $g$ and $J$ to
$T_\C M$ in the $\C$-linear way. Hence for any $X,Y\in T_\C M$, we
still have $ g(JX,JY)=g(X,Y)$  and  the Riemannian metric $g$ is represented
by
$$ds_g^2=g_{i\ell}dx^i\ts dx^\ell+g_{iL}dx^i\ts dx^L+g_{I\ell} dx^I\ts dx^\ell+g_{IJ}dx^I\ts dx^J,$$
where the metric components $g_{i\ell}, g_{iL}, g_{I\ell}$ and
$g_{IL}$ are defined in the obvious way by using local  coordinates $\left\{\frac{\p}{\p x^i}, \frac{\p}{\p x^I}\right\}$. The symmetric $\C$-bilinear  form $g:T_\C M\times T_\C M\>\C$ has the property that $$ g\left(\frac{\p}{\p
	z^i},\frac{\p}{\p z^j}\right)=0, \qtq{and} g\left(\frac{\p}{\p \bar z^i},\frac{\p}{\p \bar
	z^j}\right)=0$$ Hence, it can be regarded  as a Hermitian form  $$h: T^{1,0}M\times T^{0,1}M\>\C,\ \ \ h=h_{i\bar j}dz^i\ts d\bar z^j$$ where $h_{i\bar j}=g\left(\frac{\p}{\p
	z^i},\frac{\p}{\p \bar z^j}\right)$.\\

We also extend  the
Bismut-Strominger connection $\nabla^\MS$ and torsion $T^\MS$ to $T_{\C}M$ in the $\C$-linear way.
Hence for any $a,b,c,d\in \C$ and $X,Y,Z,W\in T_\C M$,  $$
R^\MS(aX,bY,cZ, dW)=abcd\cdot R^\MS(X,Y,Z,W).$$
We also use the components of the complexified
curvature tensor $R^\S$, for example, $$ R^\MS_{i\bar j k\bar
	\ell}:=R^\MS\left(\frac{\p}{\p z^i}, \frac{\p}{\p \bar z^j},
\frac{\p}{\p z^k}, \frac{\p}{\p \bar z^\ell}\right),$$ and
in particular we use the following notation for the complexified
curvature tensor: $$ R^\MS_{i j k \ell}:=R^\MS\left(\frac{\p}{\p z^i},
\frac{\p}{\p z^j}, \frac{\p}{\p z^k}, \frac{\p}{\p z^\ell}\right).$$
In  local holomorphic coordinates $\{z^1,\cdots, z^n\}$ on $M$,
the complexified Christoffel symbols $\ ^\MLC\Gamma$ of  the Levi-Civita connection $\nabla^\MLC$ are given by
\beq \nabla^\MLC_{\frac{\p}{\p z^A}}\frac{\p}{\p z^B}=\ ^\MLC\Gamma_{AB}^C\frac{\p}{\p z^C}\eeq
where \beq
\ ^\MLC\Gamma_{AB}^C=\sum_{E}\frac{1}{2}g^{CE}\big(\frac{\p g_{AE}}{\p
	z^B}+\frac{\p g_{BE}}{\p z^A}-\frac{\p g_{AB}}{\p
	z^E}\big)=\sum_{E}\frac{1}{2}h^{CE}\big(\frac{\p h_{AE}}{\p
	z^B}+\frac{\p h_{BE}}{\p z^A}-\frac{\p h_{AB}}{\p z^E}\big)
\eeq
where $A,B,C,E\in \{1,\cdots,n,\bar{1},\cdots,\bar{n}\}$ and
$z^{A}=z^{i}$ if $A=i$, $z^{A}=\bar{z}^{i}$ if $A=\bar{i}$.
It is easy to see that
\beq     \ ^\MLC\Gamma_{ij}^k=\frac{1}{2}h^{k\bar\ell}\left(\frac{\p h_{i\bar\ell}}{\p z^j}+\frac{\p h_{j\bar\ell}}{\p z^i}\right), \    \ ^\MLC\Gamma_{i\bar j}^k=\ ^\MLC\Gamma_{\bar j i}^k=\frac{1}{2}h^{k\bar\ell}\left(\frac{\p h_{i\bar\ell}}{\p \bar z^j}-\frac{\p h_{i\bar j}}{\p \bar z^\ell}\right),\  \ ^\MLC\Gamma_{\bar i\bar j}^k=0\eeq
Similarly, the  complexified Christoffel symbols $\ ^\MS\Gamma$ of  $\nabla^\MS$  are
\beq \nabla^\MS_{\frac{\p}{\p z^A}}\frac{\p}{\p z^B}=\ ^\MS\Gamma_{AB}^C\frac{\p}{\p z^C}.\eeq
The formula \eqref{defnSB1} tells that 
\beq \ ^\MS\Gamma_{AB}^C=\ ^\MLC\Gamma_{AB}^C+\frac{1}{2}{^{\MS}T}_{AB}^C.\eeq
Moreover, by  formula \eqref{torsion1}, one deduces that the torsion tensor has two generators
\beq {^{\MS}T}_{ij\ell}=\left(d\omega_g\right)\left(J\frac{\p}{\p z^i}, J\frac{\p}{\p z^j}, J\frac{\p}{\p z^\ell}\right)=0, \quad {^{\MS}T}_{ij\bar \ell}=\frac{\p h_{i\bar{\ell}}}{\p z^j}-\frac{\p h_{j\bar{\ell}}}{\p z^i}. \eeq 
This is equivalent to   \beq    {^{\MS}T}_{ij}^k=h^{k\bar{\ell}}\left(\frac{\p h_{i\bar{\ell}}}{\p z^j}-\frac{\p h_{j\bar{\ell}}}{\p z^i}\right), \quad {^{\MS}T}_{\bar{ij}}^{k}=0,\label{torsionformula} \eeq
and
\beq \ ^\MS\Gamma_{ij}^k=h^{k\bar{\ell}}\frac{\p h_{i\bar{\ell}}}{\p z^j},\quad ^\MS\Gamma_{\bar j i}^k=2\ ^\MLC\Gamma_{\bar j i}^k=h^{k\bar\ell}\left(\frac{\p h_{i\bar\ell}}{\p \bar z^j}-\frac{\p h_{i\bar j}}{\p \bar z^\ell}\right),\quad ^\MS\Gamma_{i\bar j}^k=\ ^\MS\Gamma_{\bar i\bar j}^k=0.\label{Christoffel} \eeq
On the other hand,  the condition $\nabla^\MS J=0$ implies that $\nabla^\MS$ preserves both $T^{1,0}M$ and $T^{0,1}M$. Hence,
\beq R^\MS_{ijk\ell}= R^\MS_{i\bar j k\ell}=R^\MS_{\bar ijk\ell}=R^\MS_{\bar i\bar jk\ell}=0, \quad  R^\MS_{ij\bar k\bar \ell}= R^\MS_{i\bar j \bar k\bar \ell}=R^\MS_{\bar ij\bar k\bar \ell}=R^\MS_{\bar i\bar j\bar k\bar \ell}=0.\label{vanishing}\eeq

\begin{lemma}\label{lem-curvature relation}
	Let $(M,\omega)$ be a Hermitian manifold. For any $X,Y\in \Gamma(M, T_\C M)$, the complexification of the real Ricci curvature of $\nabla^\MS$ defined in \eqref{realSB} is
	\beq 	\Ric^{\MS,\C}\left(X,Y\right)=h^{i\bar{\ell}}R^{\MS}\left(\pzi,X, Y,\bpzl\right)+h^{\ell\bar{i}}R^{\MS}\left(\bpzi,X,Y,\pzl\right). \label{realcomplexfication}\eeq 
	In particular, 
	\beq 	\Ric^{\MS,\C}\left(\frac{\p }{\p z^p},\frac{\p}{\p\bar z^q}\right)=h^{\ell\bar{i}} R^{\MS}_{\bar i p\bar q \ell}=h^{\ell\bar{i}}R^\MS_{p\bar i \ell\bar q}. \eeq 
\end{lemma}
\begin{proof} The first identity follows from a complexification process. Indeed,
	\be 
	\Ric^{\MS,\C}\left(X,Y\right)
	&=&\sum_{i,\ell=1}^n \left(g^{i\ell}R^{\MS}\left(\pxi,X,Y,\pxl\right)+g^{i,\ell+n}R^{\MS}\left(\pxi,X,Y,\frac{\p}{\p x^{\ell+n}}\right)\right.\\
	&&\left.+g^{i+n,\ell}R^{\MS}\left(\frac{\p}{\p x^{i+n}},X,Y,\pxl\right)+g^{i+n,\ell+n}R^{\MS}\left(\frac{\p}{\p x^{i+n}},X,Y,\frac{\p}{\p x^{\ell+n}}\right)\right).\ee 
	By using complexification relations
	$$\frac{\p}{\p x^i}=\pzi+\bpzi, \quad \frac{\partial}{\partial x^{i+n}}=\sqrt{-1}\left(\frac{\partial}{\partial z^i}-\frac{\partial}{\partial \bar{z}^i}\right), \quad h^{i\bar
		j}=2\left(g^{ij}-\sq g^{i,n+j}\right)$$
	and   similar computations as in \cite[Lemma~7.1]{LY16}	one gets \eqref{realcomplexfication}. The second identity follows from the fact that $R^{\MS}_{ij\bar k\bar\ell}=0$.
\end{proof}

\bproposition\label{curvaturecomparison} Let $(M,\omega)$ be a  balanced Hermitian manifold, i.e. $d\omega^{n-1}=0$. For any real vector field $X=X^i\frac{\p}{\p z^i}+\bar{X}^i\frac{\p}{\p \bar z^i}\in \Gamma(M,T_\R M)$, we have 
\beq \operatorname{Ric}^{\MS}(X, X)= 2\left(h^{i\bar \ell} R^{\MS}_{i\bar j k\bar\ell}\right)X^{k}\bar X^j. \label{realricci}\eeq 
\eproposition 

\bproof 
Let $X, Y\in \Gamma(M,T_\R M)$. One can write them as
$$
X=X^{1,0}+X^{0,1} \quad  \quad Y=Y^{1,0}+Y^{0,1}
$$
where $X^{1,0}, Y^{1,0}\in \Gamma(M, T^{1,0}M)$ and $X^{0,1}=\bar{X}^{1,0}$, $Y^{0,1}=\bar Y^{1,0}$. By the relation \eqref{realcomplexfication} and \eqref{vanishing},   we deduce that
\begin{eqnarray} 
&&\nonumber\operatorname{Ric}^{\MS}(X, Y)=\operatorname{Ric}^{\MS,\C}(X, Y)\\
&= &\nonumber \sum_{i, \ell=1}^n\left[h^{i \bar{\ell}} R^{\MS}\left(\frac{\partial}{\partial z^i}, X^{1,0}+X^{0,1}, Y^{1,0}, \frac{\partial}{\partial \bar{z}^{\ell}}\right)+h^{\ell \bar{i}} R^{\MS}\left(\frac{\partial}{\partial \bar{z}^i}, X^{1,0}+X^{0,1}, Y^{0,1}, \frac{\partial}{\partial z^{\ell}}\right)\right]\\ &= & \sum_{i, \ell=1}^n 2 \operatorname{Re}\left[h^{i\bar{\ell}}\left(R^{\MS}\left(\frac{\partial}{\partial z^i}, X^{0,1}, Y^{1,0}, \frac{\partial}{\partial \bar{z}^{\ell}}\right)+R^{\MS}\left(\frac{\partial}{\partial z^i}, X^{1,0}, Y^{1,0}, \frac{\partial}{\partial \bar{z}^{\ell}}\right)\right)\right]. \end{eqnarray}
In particular, we obtain 
\beq \operatorname{Ric}^{\MS}(X, X)= 2\mathrm{Re}\left[\left(h^{i\bar \ell} R^{\MS}_{i\bar j k\bar\ell}\right)X^{k}\bar X^j +\left(h^{i\bar \ell} R^{\MS}_{i j k\bar\ell}\right)X^{k} X^j\right]. \eeq 
Since $\ ^\MS\Gamma_{ i j}^{\bar k}=0$, by complexification of the curvature formula, we have
\beq h^{i\bar \ell} R^{\MS}_{i j k\bar\ell}=\frac{\p{ }^\MS\Gamma_{
		jk}^i}{\p z^i}-\frac{\p { }^\MS\Gamma_{ik}^{i}}{\p
	z^j}-{ }^\MS\Gamma^{i}_{jp}{ }^\MS\Gamma^{p}_{ik}+{ }^\MS\Gamma^{i}_{
	ip}{ }^\MS\Gamma^{p}_{jk}.\label{balance1}\eeq 
Moreover, by using \eqref{Christoffel},
\beq \frac{\p{ }^\MS\Gamma_{
		jk}^i}{\p x^i}=-{ }^\MS\Gamma_{pi}^i{ }^\MS\Gamma_{jk}^p+h^{i\bar\ell}\frac{\p^2 h_{j\bar\ell}}{\p z^i\p z^k}, \quad \frac{\p{ }^\MS\Gamma_{ik}^{i}}{\p
	x^j}=-{ }^\MS\Gamma_{pj}^i{ }^\MS\Gamma_{ik}^p+h^{i\bar\ell}\frac{\p^2 h_{i\bar\ell}}{\p z^j\p z^k}.  \label{balance2}\eeq 
Similarly, by \eqref{torsionformula}, 
\beq \frac{\p { }^\MS T_{ij}^{i}}{\p
	x^k}=-{ }^\MS\Gamma_{pk}^i{ }^\MS\Gamma_{ij}^p+{ }^\MS\Gamma_{pk}^i{ }^\MS\Gamma_{ji}^p+h^{i\bar\ell}\frac{\p^2 h_{i\bar\ell}}{\p z^j\p z^k}-  h^{i\bar\ell}\frac{\p^2 h_{j\bar\ell}}{\p z^k\p z^i}. \label{balance3}\eeq 
By  \eqref{balance1}, \eqref{balance2} and \eqref{balance3}, we obtain
\beq h^{i\bar \ell} R^{\MS}_{i j k\bar\ell}= -\frac{\p { }^\MS T_{ij}^{i}}{\p
	x^k}+{ }^\MS T_{ip}^i{ }^\MS\Gamma^{p}_{jk}.\eeq 
On the other hand, by the Bochner-Kodaira formula $[\bp^*, L]=\sqrt{-1}(\partial+[\Lambda, \partial \omega])$, 
\beq 
\bar{\partial}^* \omega  =\sqrt{-1} \Lambda(\partial \omega) =\sqrt{-1} h^{s \bar{q}}\left(\frac{\partial h_{s \bar{q}}}{\partial z^k}-\frac{\partial h_{k \bar{q}}}{\partial z^s}\right) d z^k =\sqrt{-1}{ }^\MS T_{s k}^s d z^k .
\eeq 
Here we assume $M$ is compact. In general,  if $d\omega^{n-1}=0$, by using $*\p\omega^{n-1}=0$,  we obtain ${ }^\MS T_{s k}^s=0$ for each $k$ and so
\beq h^{i\bar \ell} R^{\MS}_{i j k\bar\ell}=0.\eeq 
By using a similar computation, one can show that
$\left( h^{i\bar \ell} R^{\MS}_{i \bar j k\bar\ell}\right)$ 
is a Hermitian matrix (see also \cite[Corollary~1.8]{WY25} on the formula of $	\mathfrak{R}ic^{(3)}(\omega, 1)$). Therefore, we obtain \eqref{realricci}.
\eproof

   \vskip 2\baselineskip
   \section{Index forms on Hermitian manifolds}

     In this section, we formulate the second variation formula and the index form by using the Strominger-Bismut connection $\nabla^\MS$. In particular, we prove Theorem \ref{thm-2nd} and Theorem \ref{thm-index form}.\\

   Let $\gamma:[a,b]\to M$ is a unit-speed geodesic in a Riemannian manifold $(M,g)$. If   $\alpha:[a,b]\times(-\eps_1,\eps_1)\times(-\eps_2,\eps_2)\to M$ is  a variation of $\gamma(t)$ with variational vector fields $V$ and  $W$, i.e.
   $
   \alpha(t,0,0)=\gamma(t)$
   and
   $$
   \alpha_*\left(\frac{\p}{\p s_1}\right)\bigg|_{s_1=s_2=0}=V,\ \ \alpha_*\left(\frac{\p}{\p s_2}\right)\bigg|_{s_1=s_2=0}=W.
   $$
   It is well-known that, the second variation of the energy of $\gamma$ is given by
   \begin{eqnarray} \frac{\p^2  E(\alpha)}{\p s_1\p s_2}\bigg|_{s_1=s_2=0}\nonumber&
   =&\int_a^b\left\{\LL\hat{\nabla}^{\text{LC}}_{\frac{d}{dt}}V,\hat{\nabla}^{\text{LC}}_{\frac{d}{dt}}W\RL-R^{\text{LC}}(V,\gamma',\gamma',W)\right\}dt\\&&+\left.\LL\left.\left(\bar{\nabla}^\MLC_{\frac{\p}{\p s_1}}\alpha_*\left(\frac{\p}{\p s_2}\right)\right)\right|_{s_1=s_2=0},\gamma'\RL\right|_{t=a}^{t=b},\label{secondvariation}
   \end{eqnarray} where $\hat\nabla^{\MLC}$  and $\overline\nabla^{\MLC}$ are the induced Levi-Civita connections on $\gamma^*TM$ and $\alpha^*TM$ respectively. Recall that the index form $I_\gamma: \Gamma(\gamma^*TM)\times \Gamma(\gamma^*TM)\>\R$ is given by
   \beq I_\gamma(V,W) =\int_a^b\left\{\LL\hat{\nabla}^{\text{LC}}_{\frac{d}{dt}}V,\hat{\nabla}^{\text{LC}}_{\frac{d}{dt}}W\RL-R^{\text{LC}}(V,\gamma',\gamma',W)\right\}dt. \eeq
   When $\alpha$ is a proper second variation of $\gamma$, i.e. $\alpha(a, s_1,s_2)\equiv\gamma(a)$, $\alpha(b, s_1,s_2)\equiv \gamma(b)$,  one has
   \beq     \frac{\p^2  E(\alpha)}{\p s_1\p s_2}\bigg|_{s_1=s_2=0}
   =\int_a^b\left\{\LL\hat{\nabla}^{\text{LC}}_{\frac{d}{dt}}V,\hat{\nabla}^{\text{LC}}_{\frac{d}{dt}}W\RL-R^{\text{LC}}(V,\gamma',\gamma',W)\right\}dt. \eeq

   \vskip 1\baselineskip

   \bproof[Proof of Theorem \ref{thm-2nd}]
   Since $\bar\nabla^\MS$ is compatible with the induced metric on $\alpha^*TM$,  a straightforward computation shows
   \be  \frac{1}{2}\frac{\p^2}{\p s_1\p s_2}\left| \alpha_*\left(\frac{\p}{\p t}\right)\right|^2=\frac{\p}{\p s_1}\LL \bar{\nabla}_{\frac{\p}{\p s_2}}^{\MS}\alpha_*\left(\frac{\p}{\p t}\right),\alpha_*\left(\frac{\p}{\p t}\right)\RL=\frac{\p}{\p s_1}\LL \bar{\nabla}_{\frac{\p}{\p t}}^{\MS}\alpha_*\left(\frac{\p}{\p s_2}\right),\alpha_*\left(\frac{\p}{\p t}\right)\RL \ee
   where the second identity follows from the relation
   $$\LL \bar{\nabla}_{\frac{\p}{\p s_2}}^{\MS}\alpha_*\left(\frac{\p}{\p t}\right),\alpha_*\left(\frac{\p}{\p t}\right)\RL=\LL \bar{\nabla}_{\frac{\p}{\p t}}^{\MS}\alpha_*\left(\frac{\p}{\p s_2}\right),\alpha_*\left(\frac{\p}{\p t}\right)\RL$$
   since  the torsion term $T\left(\alpha_*\left(\frac{\p}{\p s_2}\right),\alpha_*\left(\frac{\p}{\p t}\right) ,\alpha_*\left(\frac{\p}{\p t}\right)\right)=0$. Hence,
   \be \frac{\p}{\p s_1}\LL \bar{\nabla}_{\frac{\p}{\p t}}^{\MS}\alpha_*\left(\frac{\p}{\p s_2}\right),\alpha_*\left(\frac{\p}{\p t}\right)\RL &=&\LL \bar{\nabla}_{\frac{\p}{\p s_1}}^{\MS}\bar{\nabla}_{\frac{\p}{\p t}}^{\MS}\alpha_*\left(\frac{\p}{\p s_2}\right),\alpha_*\left(\frac{\p}{\p t}\right)\RL\\&&+\LL\bar{\nabla}_{\frac{\p}{\p t}}^{\MS}\alpha_*\left(\frac{\p}{\p s_2}\right),\bar{\nabla}_{\frac{\p}{\p s_1}}^{\MS}\alpha_*\left(\frac{\p}{\p t}\right)\RL. \ee
   By using the curvature formula of $R^\MS$, one has
   $$\bar{\nabla}_{\frac{\p}{\p s_1}}^{\MS}\bar{\nabla}_{\frac{\p}{\p t}}^{\MS}\alpha_*\left(\frac{\p}{\p s_2}\right)= \bar{\nabla}_{\frac{\p}{\p t}}^{\MS}\bar{\nabla}_{\frac{\p}{\p s_1}}^{\MS}\alpha_*\left(\frac{\p}{\p s_2}\right)+R^{\MS}\left(\alpha_*\left(\frac{\p}{\p s_1}\right),\alpha_*\left(\frac{\p}{\p t}\right)\right)\alpha_*\left(\frac{\p}{\p s_2}\right).$$
   Therefore,
   \be \frac{1}{2}\frac{\p^2}{\p s_1\p s_2}\left| \alpha_*\left(\frac{\p}{\p t}\right)\right|^2
   &=&\LL \bar{\nabla}_{\frac{\p}{\p t}}^{\MS}\bar{\nabla}_{\frac{\p}{\p s_1}}^{\MS}\alpha_*\left(\frac{\p}{\p s_2}\right),\alpha_*\left(\frac{\p}{\p t}\right)\RL+\LL\bar{\nabla}_{\frac{\p}{\p t}}^{\MS}\alpha_*\left(\frac{\p}{\p s_2}\right),\bar{\nabla}_{\frac{\p}{\p t}}^{\MS}\alpha_*\left(\frac{\p}{\p s_1}\right)\RL\\&&+R^{\MS}\left(\alpha_*\left(\frac{\p}{\p s_1}\right),\alpha_*\left(\frac{\p}{\p t}\right),\alpha_*\left(\frac{\p}{\p s_2}\right),\alpha_*\left(\frac{\p}{\p t}\right)\right)\\
   &&+T^\MS\left(\alpha_*\left(\frac{\p}{\p s_1}\right),\alpha_*\left(\frac{\p}{\p t}\right),\bar{\nabla}_{\frac{\p}{\p t}}^{\MS}\alpha_*\left(\frac{\p}{\p s_2}\right)\right). \ee
   Note also that the first term on the right hand side can be written as
   \be \LL \bar{\nabla}_{\frac{\p}{\p t}}^{\MS}\bar{\nabla}_{\frac{\p}{\p s_1}}^{\MS}\alpha_*\left(\frac{\p}{\p s_2}\right),\alpha_*\left(\frac{\p}{\p t}\right)\RL&=&\frac{\p}{\p t}\LL\bar{\nabla}^\MS_{\frac{\p}{\p s_1}}\alpha_*\left(\frac{\p}{\p s_2}\right),\alpha_*\left(\frac{\p}{\p t}\right)\RL\\&&-\LL\bar{\nabla}_{\frac{\p}{\p s_1}}^{\MS}\alpha_*\left(\frac{\p}{\p s_2}\right),\bar{\nabla}_{\frac{\p}{\p t}}^{\MS}\alpha_*\left(\frac{\p}{\p t}\right)\RL.\ee
   When $s_1=s_2=0$,
   $$\left.\bar{\nabla}_{\frac{\p}{\p t}}^{\MS}\alpha_*\left(\frac{\p}{\p t}\right)\right|_{s_1=s_2=0}=\hat{\nabla}^\MS_{\frac{d}{dt}}\gamma'=0,$$
   $$\left.\LL\bar{\nabla}_{\frac{\p}{\p t}}^{\MS}\alpha_*\left(\frac{\p}{\p s_2}\right),\bar{\nabla}_{\frac{\p}{\p t}}^{\MS}\alpha_*\left(\frac{\p}{\p s_1}\right)\RL\right|_{s_1=s_2=0}=\LL\hat{\nabla}^{\MS}_{\frac{d}{d t}}V,\hat{\nabla}^{\MS}_{\frac{d}{d t}}W\RL,$$
   and 
   $$\left.T^\MS\left(\alpha_*\left(\frac{\p}{\p s_1}\right),\alpha_*\left(\frac{\p}{\p t}\right),\bar{\nabla}_{\frac{\p}{\p t}}^{\MS}\alpha_*\left(\frac{\p}{\p s_2}\right)\right)\right|_{s_1=s_2=0}=T^{\MS}\left(V,\gamma',\hat{\nabla}^{\MS}_{\frac{d}{d t}}W\right).$$
    Moreover, by using \eqref{skewsymmetry},
    \beq\left. R^{\MS}\left(\alpha_*\left(\frac{\p}{\p s_1}\right),\alpha_*\left(\frac{\p}{\p t}\right),\alpha_*\left(\frac{\p}{\p s_2}\right),\alpha_*\left(\frac{\p}{\p t}\right)\right)\right|_{s_1=s_2=0}=-R^\MS(V,\gamma',\gamma', W),\eeq 
    and
     one obtains the desired variation formula \eqref{2nd}. If  $\alpha$ is a proper variation of $\gamma$, one has $\alpha(a, s_1,s_2)=\gamma(a)$, $\alpha(b, s_1,s_2)=\gamma(b)$, and so
   $$\left(\bar{\nabla}^\MS_{\frac{\p}{\p s_1}}\alpha_*\left(\frac{\p}{\p s_2}\right)\right)(a,s_1,s_2)=\left(\bar{\nabla}^\MS_{\frac{\p}{\p s_1}}\alpha_*\left(\frac{\p}{\p s_2}\right)\right)(b,s_1,s_2)=0.$$
   Hence, \eqref{normalsecondvariation} follows from \eqref{2nd}.
   \eproof

   \begin{proof}[Proof of Theorem \ref{thm-index form}]
    By  Theorem \ref{thm-2nd} and formula \eqref{secondvariation},  for any smooth vector fields $V$ and $W$,
    \be
    I_\gamma(V,W)&=&\frac{\p^2E(\alpha)}{\p s_1\p s_2}\bigg|_{s_1=s_2=0}-\left.\LL\left.\left(\bar{\nabla}^\MLC_{\frac{\p}{\p s_1}}\alpha_*\left(\frac{\p}{\p s_2}\right)\right)\right|_{s_1=s_2=0},\gamma'\RL\right|_{t=a}^{t=b}\\
    &=&\int_a^b\left\{\LL\hat{\nabla}^{\MS}_{\frac{d}{d t}}V,\hat{\nabla}^{\MS}_{\frac{d}{d t}}W\RL+T^{\MS}\left(V,\gamma',\hat{\nabla}^{\MS}_{\frac{d}{d t}}W\right)-R^{\MS}(V,\gamma',\gamma',W)\right\}dt\\
    &&+\left.\LL\left.\left(\bar{\nabla}^\MS_{\frac{\p}{\p s_1}}\alpha_*\left(\frac{\p}{\p s_2}\right)\right)\right|_{s_1=s_2=0},\gamma'\RL\right|_{t=a}^{t=b}-\left.\LL\left.\left(\bar{\nabla}^\MLC_{\frac{\p}{\p s_1}}\alpha_*\left(\frac{\p}{\p s_2}\right)\right)\right|_{s_1=s_2=0},\gamma'\RL\right|_{t=a}^{t=b}.\ee
    On the other hand, by Lemma \ref{definition of Strominger connection}, the last line equals
    \beq \LL\bar{\nabla}^{\MS}_{\frac{\p}{\p s_1}}\alpha_*\left(\frac{\p}{\p s_2}\right)-\bar{\nabla}^{\MLC}_{\frac{\p}{\p s_1}}\alpha_*\left(\frac{\p}{\p s_2}\right)\bigg|_{s_1=s_2=0},\gamma'\RL\bigg|_{t=a}^{t=b}=\frac{1}{2}T^{\MS}(V,W,\gamma')\bigg|_{t=a}^{t=b}.
    \eeq
    Hence, we obtain the desired expression.
   \end{proof}

   \vskip 1\baselineskip

   \section{Proofs of comparison theorems}

   In this section we prove Theorem \ref{Myers},  Theorem \ref{thm-compactness from HSC} and Theorem \ref{cor-auto}.

   \bproof[Proof of Theorem \ref{Myers}]  Suppose, by contradiction, that $\gamma:[0,L]\to M$ is a unit speed minimal geodesic with length $L>\pi/\sqrt{K}$.
   Let $\alpha:[a,b]\times(-\eps_1,\eps_1)\times(-\eps_2,\eps_2)\to M$ be a \textbf{proper} second variation of $\gamma$ with variational vector fields
   \beq
   \alpha_*\left(\frac{\p}{\p s_1}\right)\bigg|_{s_1=s_2=0}=\alpha_*\left(\frac{\p}{\p s_2}\right)\bigg|_{s_1=s_2=0}=V.
   \eeq
   Let $\{e_1,\ldots,e_n,e_{n+1},\ldots,e_{2n}\}$ be an orthonormal basis for $T_{\gamma(0)}M$ such that $Je_i=e_{i+n},e_1=\gamma'(0)$.
   Consider $(2n-1)$ variation fields
   $$
   V_j(t)=f(t)\cdot P^{\text{SB}}_{0,t;\gamma}e_j
   $$ for $2\leq j\leq 2n$  and $f(0)=f(L)=0$.
   Then
   $$
   \hat{\nabla}^{\text{SB}}_{\frac{d}{d t}}V_j=f'(t)P^{\text{SB}}_{0,t;\gamma}e_j.
   $$
 By Theorem \ref{thm-2nd}, we obtain
   \be
   0&\leq&\sum_{j=2}^{2n}\int_0^L\bigg(\LL\hat{\nabla}^\MS_{\frac{d}{d t}}V_j,\hat{\nabla}^{\MS}_{\frac{d}{d t}}V_j\RL+T\left(V_j,\gamma',\hat{\nabla}^{\MS}_{\frac{d}{d t}}V_j\right)-R^{\MS}(V_j,\gamma',\gamma',V_j)\bigg)dt\\
   &=&\int_0^L (2n-1)(f'(t))^2-(f(t))^2\Ric^{\MS}\left(\gamma',\gamma'\right)dt.
   \ee
   Suppose that  \beq \mathfrak{Ric}^{\MS} (V,\bar V)\geq (2n-1)K |V|^2, \eeq 
   for each $V\in\Gamma(M,T^{1,0}M)$, then by Proposition \ref{curvaturecomparison},
   \beq \Ric^{\MS}\left(\gamma',\gamma'\right)\geq (2n-1)K.\eeq 
   If we choose $f(t)=\sin(\pi t/L)$, then
   \be
   0&\leq& (2n-1)\int_0^L \left(\frac{\pi^2}{L^2}\cos^2(\pi t/L)-K\sin^2(\pi t/L)\right)dt\\
   &=&(2n-1)\int_0^L\sin^2(\pi t/L)\left(\frac{\pi^2}{L^2}-K\right)dt<0.
   \ee
 This is a contradiction.  Therefore we conclude that
   $ L\leq\pi/\sqrt{K}$ and  $M$ is compact. By considering the universal of $M$, we deduce that $\pi_1(M)$ is finite. \\
   
   We give a sketched proof of the Laplacian comparison theorem under the curvature condition \eqref{complexcurvature}.
   Suppose $q\in M\setminus(\mathrm{cut}(p)\cup\{p\})$. Let $\gamma:[0,b]\to M$ be the unit speed minimal geodesic such that $\gamma(0)=p,\gamma(b)=q$.
   Let $\{e_1,\ldots,e_n,e_{n+1},\ldots,e_{2n}\}$ be an orthonormal basis for $T_pM$ such that $Je_i=e_{i+n},e_1=\gamma'(0)$.
   Define the parallel transport of $e_i$ along $\gamma$ with respect to the Strominger-Bismut connection $\nabla^\MS$
   $$
   e^\MS_i(t)=P_{0,t;\gamma}^{\MS}e_i,\ \ \  1\leq i\leq 2n.
   $$
   Let $X_i(t)$ be the Jacobi field along $\gamma$ such that $X_i(0)=0,X_i(b)=e^\MS_i(b)$, i.e.
   $$\hat\nabla^{\MLC}_{\frac{d}{dt}}\hat\nabla^{\MLC}_{\frac{d}{dt}} X_i+R^{\MLC}(X_i,\gamma')\gamma'=0.$$
   It is easy to see that
   $X_1(t)=\frac{t \gamma'}{b}$ and $\LL X_i,\gamma'\RL=0$ for $i\geq 2$.
   It is well-known that for any normal Jacobi
   field $X$ along $\gamma$ with $X(0)=0$ and
   any vector field $W(t)$ along $\gamma$ with $W(0)=0$, one has
   \begin{eqnarray} \left( \mathrm{Hess}\ r\right) (X(s), W(s))\nonumber&=&I_\gamma (X(s), W(s))\\ \nonumber&=&\int_0^s \left(\left\la \hat\nabla^{\MLC}_{\frac{d}{dt}} X(t), \hat\nabla^{\MLC}_{\frac{d}{dt}} W(t)\right\ra- R^{\MLC}(X(t),\gamma'(t),\gamma'(t), W(t))\right)dt.\end{eqnarray}
   Since $ \left( \mathrm{Hess}\ r\right) (\gamma'(b), \gamma'(b))=0$,   one obtains
   \beq
   \Delta_g r(q)=\sum_{i=2}^{2n}I_\gamma(X_i(b),X_i(b)).
   \eeq
   We define
   $
   Z_i(t)=f(t)e^\MS_i(t)
   $
   with $f(t)=\frac{\sn_k(t)}{\sn_k(b)}$ for $i=2,\cdots, 2n$.   Then $Z_i(a)=X_i(a)=0$ and $Z_i(b)=X_i(b)$. This implies
   \beq \Delta_g r(q)=\sum_{i=2}^{2n}I_\gamma (X_i(b),X_i(b))\leq \sum_{i=2}^{2n}I_\gamma(Z_i(b),Z_i(b)).\eeq
   For $2\leq i\leq 2n$, by Theorem \ref{thm-index form},
   \be I_\gamma(Z_i(b),Z_i(b))&=&\int_0^b\left\{\LL\hat{\nabla}^{\MS}_{\frac{d}{d t}}Z_i,\hat{\nabla}^{\MS}_{\frac{d}{d t}}Z_i\RL+T^{\MS}\left(Z_i,\gamma',\hat{\nabla}^{\MS}_{\frac{d}{d t}}Z_i\right)-R^{\MS}(Z_i,\gamma',\gamma',Z_i)\right\}dt\\
   &=&\int_0^b\left(|f'(t)|^2-f^2(t)R^{\MS}(e^\MS_i,\gamma',\gamma',e^\MS_i) \right)dt.\ee
 On the other hand,  we have the estimate
   $$\sum_{i=1}^{2n} R^{\MS}(e^\MS_i,\gamma',\gamma',e^\MS_i) \geq (2n-1)K.$$
   Hence
   \beq
   \Delta_g r(q)
   \leq(2n-1)\int_0^b\left(|f'(t)|^2-Kf^2(t)\right)dt
   =(2n-1)  \frac{\sn'_K(r)}{\sn_K(r)}.\label{Laplacian}
   \eeq 
   
   \vskip 1\baselineskip
   
   It is well-established that the Bishop-Gromov volume comparison theorem follows from the Laplacian comparison theorem and Myers' diameter estimate under Ricci curvature bounds. We adopt an analogous approach in our setting, leveraging the identical geodesic structures and exponential maps between two frameworks. Let $p\in M$ and$U=M\setminus \mathrm{cut}(p)$.  Let $B(p,\delta)$ be the
   {metric ball} centered at $p$ with radius $\delta$, and $g_K$
   be the constant curvature metric.   Consider the trivialization map in the normal coordinates on the domain of injective radius
$\Sigma(p)$, 
$$\R^+\times \S^{2n-1}\stackrel{\Phi}{\>}T_pM\backslash\{0\}\cong
\R^{2n}\backslash\{0\}\>U\backslash\{p\}$$ given by $\Phi(\rho,\omega)=\rho\omega$,
\begin{eqnarray} \mathrm{Vol}_g\left(B(p,\delta)\right)&=&\nonumber\mathrm{Vol}_g\left(\exp_p(B(0,\delta)\cap \Sigma(p))\right)\\
&=&\int_{\S^{2n-1}}\int_0^\delta \chi_{\Sigma(p)}\cdot \sqrt{\det g}
\circ \Phi(\rho,\omega)\cdot  \rho^{2n-1}\cdot  d\rho
d{\mathrm{vol}}_{\S^{2n-1}}.
\end{eqnarray}
For each fixed $\omega\in\S^{n-1}$
the volume density ratio is defined as \beq \lambda(\rho,\omega)=
\frac{\rho^{2n-1}\sqrt{\det g}\circ \Phi(\rho,
	\omega)}{\sn_K^{2n-1}(\rho)}. \eeq
By using the Laplacian comparison result \eqref{Laplacian}, one has \beq \p_r \log \left( r^{2n-1}\sqrt{\det g}\right)=\Delta_g r\geq (2n-1)\frac{\sn_K'(r)}{\sn_K(r)}=\p_r\log\left(\sn_K^{2n-1}(r)\right). \eeq 
Hence, the volume density ratio
$\lambda(\rho,\omega)$ is  decreasing in $\rho\in (0,\delta)$ and $
\lim\limits_{\rho\>0}\lambda(\rho,\omega)=1$.
 If $K=1/R^2>0$,
$$ \frac{\mathrm{Vol}_g(B(p,\delta))}{\mathrm{Vol}_{g_k}(B(p,\delta))}=\frac{1}{\mathrm{Vol}(\S^{2n-1})}\int_{\S^{2n-1}}\left(\frac{\int_0^\delta \chi_{\Sigma(p)}\cdot \rho^{2n-1}\sqrt{\det g}\circ \Phi(\rho, \omega)d\rho}{\int_0^\delta\chi_{B(0,\pi R)}\cdot \sn_K^{2n-1}(\rho) d\rho}\right)d\mathrm{vol}_{\S^{2n-1}}.$$
By first part of Theorem \ref{Myers}, $\diam(M,g)\leq \pi R$.  Therefore, $\Sigma(p)\subset B(0,\pi R)$ and $$\tilde  \lambda(\rho,\omega)= \frac{\chi_{\Sigma(p)}\cdot\rho^{2n-1}\sqrt{\det g}\circ \Phi(\rho, \omega)}{\chi_{B(0,\pi R)}\cdot\sn_K^{2n-1}(\rho)} $$ is nonincreasing in $\rho$. Hence $
\frac{\mathrm{Vol}_g(B(p,\delta))}{\mathrm{Vol}_{g_K}(B(p,\delta))}$
is nonincreasing in $\rho$. On ther other hand,
$\lim\limits_{\delta\>0}\frac{\mathrm{Vol}_g(B(p,\delta))}{\mathrm{Vol}_{g_K}(B(p,\delta))}=1$, 
we deduced that $
\mathrm{Vol}_g(B(p,\delta))\leq \mathrm{Vol}_{g_K}(B(p,\delta))$.   When $\delta$ approaches $\pi/\sqrt{K}$, one deduces that $ \mathrm{Vol}(M,\omega)\leq \mathrm{Vol}\left(\mathbb S^{2n}\left(1/\sqrt{K}\right), g_{\mathrm{can}}\right)$.
\eproof 

\bremark It is obvious that Theorem \ref{Myers} also holds if the curvature condition \eqref{complexcurvature} is replaced by $$\Ric^{\MS} (X,X)\geq (2n-1)K |X|^2_g,\quad  X\in T_\R M.$$
\eremark 
  
  \bproof[Proof of Theorem \ref{thm-compactness from HSC}]   Suppose $\gamma:[0,L]\to M$ is a unit speed minimal geodesic with length $L>\pi/\sqrt{K}$.
   Let $\alpha:[a,b]\times(-\eps_1,\eps_1)\times(-\eps_2,\eps_2)\to M$ be a \textbf{proper} second variation of $\gamma$ with variational vector fields $V$ and $W$.
   In particular, we choose  \beq V(t)=W(t)=\sin\left(\frac{\pi t}{L}\right)J\gamma'(t). \eeq
   where $J$ is the complex structure of $(M,\omega)$.    A simple calculation shows that
  \beq
   \hat{\nabla}^{\MS}_{\frac{d}{d t}}\left(J\gamma'(t)\right)=J\left(\hat{\nabla}^{\MS}_{\frac{d}{d t}}\gamma'(t)\right)=0.
\eeq
   Hence, by Theorem \ref{thm-2nd}, we obtain
   $$0\leq \frac{\p^2 E(\alpha)}{\p s_1\p s_2}\bigg|_{s_1=s_2=0}
   =\int_0^L\left\{\LL\hat{\nabla}^{\MS}_{\frac{d}{d t}}V,\hat{\nabla}^{\MS}_{\frac{d}{d t}}W\RL+T^{\MS}\left(V,\gamma',\hat{\nabla}^{\MS}_{\frac{d}{d t}}W\right)-R^{\MS}(V,\gamma',\gamma',W)\right\}dt.$$    Moreover, since $T$ is totally skew-symmetric,
   \beq T^{\MS}\left(V,\gamma',\hat{\nabla}^{\MS}_{\frac{d}{d t}}W\right)=\frac{\pi}{L}\sin\left(\frac{\pi t}{L}\right)\cos\left(\frac{\pi t}{L}\right)T^\MS(J\gamma',\gamma', J\gamma')=0. \eeq  Therefore, we have
   \be  \frac{\p^2 E(\alpha)}{\p s_1\p s_2}\bigg|_{s_1=s_2=0}
   &=&\int_0^L\left(\frac{\pi^2}{L^2}\cos^2\left(\frac{\pi t}{L}\right)-\sin^2\left(\frac{\pi t}{L}\right)R^{\MS}(J\gamma',\gamma',\gamma',J\gamma')\right)dt\\
   &\leq&\int_0^L\left(\frac{\pi^2}{L^2}\cos^2\left(\frac{\pi t}{L}\right)-\sin^2\left(\frac{\pi t}{L}\right)K\right) dt\\&=&\int_0^L\cos^2\left(\frac{\pi t}{L}\right)\left(\frac{\pi^2}{L^2}-K\right) dt<0.
   \ee
   This is  a contradiction. Hence, the diameter of $M$ is bounded by $\pi/\sqrt{K}$.\\

   Moreover,
   suppose $\pi_1(M)$ is not trivial.
   Let $\gamma:[0, \ell]\to M$ be a unit speed closed geodesic representing a nontrivial element in $\pi_1(M)$  that has minimum length in the free homotopy class of $\gamma$.  Let $V(t)=J\gamma'(t)$ and define  a variation $\alpha:[0,\ell]\times(-\eps,\eps)\to M$ of $\gamma(t)$ as
   $$\alpha(t,s)=\exp_{\gamma(t)}(sV(t)).$$ Recall that geodesics for $\nabla^{\MLC}$ and $\nabla^\MS$ are the same and so  exponential maps of them are the same.
    It is easy to see that $\alpha$ is a smooth   variation of $\gamma$ by loops in the free homotopy class, i.e.
   \beq \alpha(t,0)=\gamma(t),\ \ \alpha(0,s)=\alpha(\ell, s),\ \ \forall s\in(-\eps,\eps).\label{loop} \eeq  Hence,
   \beq \frac{\p^2 E(\alpha)}{\p s^2}\bigg|_{s=0}\geq 0. \eeq
   By  using \eqref{2nd} in Theorem \ref{thm-2nd}, we have
   \beq \frac{\p^2 E(\alpha)}{\p s^2}\bigg|_{s=0} =-\int_0^\ell R^{\MS}(J\gamma',\gamma',\gamma',J\gamma')dt
   + \left.\LL\left.\left(\bar{\nabla}^\MS_{\frac{\p}{\p s}}\alpha_*\left(\frac{\p}{\p s}\right)\right)\right|_{s=0},\gamma'\RL\right|_{t=0}^{t=\ell}. \eeq
   By \eqref{loop}, the second term on the right hand side is zero. Hence
   $
   \frac{\p^2 E(\alpha)}{\p s^2}\bigg|_{s=0}<0
   $ and this is
   a contradiction.
   \eproof

     \noindent \bproof[Proof of  Theorem \ref{cor-auto}]    Let $F:(M,g)\>(M,g)$ be an isometry. Suppose to the contrary that $F$  has no fixed point. Then the displacement function  $\delta_F(x)=d(x,
     F(x))$ has a  minimum $\ell>0$. Suppose $\delta_F$ takes its minimum at point $p\in M$. Let $\gamma:[0,\ell]\>M$ be a minimal  unit speed geodesic connecting $p$ and $q=F(p)$. By a simple computation (e.g. \cite[Lemma~6.2.7]{Pet16}),  $\gamma$ extends to an axis of $F$, i.e. $\gamma:\R\>M$ satisfies
   \beq F(\gamma(t))=\gamma(t+\ell).\eeq 
     In particular, $\gamma'(\ell)=F_{*}\gamma'(0)$.\\

    We employ a setup similar to that in the previous paragraph.  Let $V(t)=J\gamma'(t)$ and define  a variation $\alpha:[0,\ell]\times(-\eps,\eps)\to M$ of $\gamma(t)$ by
     $$\alpha(t,s)=\exp_{\gamma(t)}(sV(t)).$$
     A straightforward computation shows
     \beq V(\ell)=J\gamma'(\ell)=J\circ F_{*p}(\gamma'(0))=F_{*p}(J\gamma'(0))=F_{*p}(V(0)). \eeq 
    Since  the geodesic is unique  with respect to the given initial values, we deduce that \beq \alpha(\ell,s)=F\circ \alpha(0,s). \eeq  By  using \eqref{2nd} in Theorem \ref{thm-2nd} again, the second variation of the
     energy of $\gamma(t)$ is
     \begin{eqnarray} \left. \frac{\p^2}{\p s^2}\right|_{s=0}E(\alpha(\bullet, s))\nonumber
     &=&\int_0^\ell\left\{\LL\hat{\nabla}^{\MS}_{\frac{d}{d t}}V,\hat{\nabla}^{\MS}_{\frac{d}{d t}}V\RL+T^{\MS}\left(V,\gamma',\hat{\nabla}^{\MS}_{\frac{d}{d t}}V\right) dt-R^{\MS}(V,\gamma',\gamma',V)\right\}dt\\
     &&+\left.\LL\left.\left(\bar{\nabla}^\MS_{\frac{\p}{\p s}}\alpha_*\left(\frac{\p}{\p s}\right)\right)\right|_{s=0},\gamma'\RL\right|_{t=0}^{t=\ell}    \\
     &=&-\int_0^\ell R^\MS(V,\gamma',\gamma', V) dt+\left.\LL\left.\left(\bar{\nabla}^\MS_{\frac{\p}{\p s}}\alpha_*\left(\frac{\p}{\p s}\right)\right)\right|_{s=0},\gamma'\RL\right|_{t=0}^{t=\ell}.\end{eqnarray}

 \noindent     If we set $c(s)=\alpha(0,s)$ and $\tilde c(s)=F(c(s))$, then
     \be \LL \left(\left. \bar{\nabla}^\MS_{\frac{\p}{\p s}} \alpha_*\left(\frac{\p}{\p s}\right)\right)\right|_{s=0,t=\ell},\gamma'(\ell) \RL&=& \LL \left(\left. \bar{\nabla}^\MS_{\frac{\p}{\p s}}\tilde c'(s)\right)\right|_{s=0},\gamma'(\ell) \RL\\
     &=&\LL F_*\left(\left(\left.\bar{\nabla}^\MS_{\frac{\p}{\p s}} c'(s)\right)\right|_{s=0}\right), F_*(\gamma'(0))\RL\\
     &=&\LL \left(\left.\bar{\nabla}^\MS_{\frac{\p}{\p s}} \alpha_*\left(\frac{\p}{\p s}\right)\right)\right|_{s=0,t=0},\gamma'(0) \RL
     \ee
     where the second identity follows form the fact that $F$ is an isometry. Hence,
     \beq \left. \frac{\p^2}{\p s^2}\right|_{s=0}E(\alpha(\bullet, s))\nonumber=-\int_0^\ell R^\MS(V,\gamma',\gamma', V) dt=-\int_0^\ell R^{\MS}(J\gamma',\gamma',\gamma',J\gamma')dt<0.\eeq
     This is a contradiction since $\gamma$ is a minimum geodesic connecting $p$ and $F(p)$.
   \eproof

\vskip 1\baselineskip


\end{document}